\documentclass[11pt]{article}
\usepackage{}
\usepackage{amsfonts}
\usepackage{amssymb}
\usepackage{amsthm}
\usepackage{amsmath}
\usepackage{graphicx}
\usepackage{empheq}
\usepackage{indentfirst}
\usepackage{cite}
\usepackage{mathrsfs}
\usepackage{graphics}
\usepackage{enumerate}

 \newtheorem{theorem}{Theorem}[section]
 
 \newtheorem{lemma}{Lemma}[section]

 \numberwithin{equation}{section}

\allowdisplaybreaks[4]

\def\e{\epsilon}

\newcommand{\beq}{\begin{equation}}
\newcommand{\eeq}{\end{equation}}
 \def\non{\nonumber }
\def\bea{\begin{eqnarray}}
\def\eea{\end{eqnarray}}

\begin{document}

\title{Global Existence and Asymptotic
 Behavior of Solutions to a Chemotaxis-Fluid System on General Bounded Domain}
\author{Jie Jiang\thanks{Wuhan Institute of Physics and Mathematics, Chinese Academy of Sciences,
Wuhan 430071, HuBei Province, P.R. China,
\textsl{jiangbryan@gmail.com}.}, \ \ Hao Wu\thanks{School of Mathematical Sciences and Shanghai Key Laboratory for Contemporary Applied
Mathematics, Fudan University, Shanghai 200433,\ P.R. China,
\textsl{haowufd@yahoo.com}.}\ \ and \ Songmu Zheng\thanks{School of Mathematical Sciences, Fudan University, Shanghai 200433,\ P.R. China,
\textsl{songmuzheng@yahoo.com}.}}
\date{\today}
\maketitle
\begin{abstract} In this paper, we investigate an initial-boundary value problem for a chemotaxis-fluid system in a general bounded regular domain $\Omega \subset \mathbb{R}^N$ ($N\in\{2,3\}$), not necessarily being convex. Thanks to the elementary lemma given by Mizoguchi \& Souplet \cite{MizoSoup}, we can derive a new type of entropy-energy estimate, which enables us to prove the following: (1) for $N=2$, there exists  a unique global classical solution to the full chemotaxis-Navier-Stokes system, which converges to a constant steady state $(n_\infty, 0,0)$ as $t\to+\infty$, and (2) for $N=3$, the existence of a global weak solution to the simplified chemotaxis-Stokes system. Our results generalize the recent work due to Winkler \cite{WinklerCPDE,WinkARMA}, in which the domain $\Omega$ is essentially assumed to be convex.
\end{abstract}

\section{Introduction}
In this paper, we study the following initial-boundary value problems for a chemotaxis(-Navier)-Stokes system  \cite{Lorz2010,PNAS2005}:
\beq\begin{cases}\label{chemo1}
n_t+u\cdot\nabla n=\Delta n-\nabla\cdot(n\chi(c)\nabla c),\\
c_t+u\cdot\nabla c=\Delta c-nf(c),\\
u_t=\Delta u+\kappa(u\cdot\nabla)u+\nabla p+n\nabla\phi,\\
\nabla\cdot u=0.\end{cases}
\eeq
The chemotaxis-fluid system \eqref{chemo1} was recently proposed in \cite{PNAS2005} to model the motion of swimming bacteria under the effects of diffusion, oxygen-taxis and transport through an incompressible fluid. Here, we assume that $\Omega\subset \mathbb{R}^N$ $(N\in\{2,3\})$ is a bounded domain with smooth boundary ${\partial\Omega}$. The scalar functions $n$ and $c$ denote the concentration of oxygen and bacteria, respectively. The vector $u$ stands for the velocity field of the fluid subject to an incompressible Navier-Stokes type equation with pressure $p$.  The scalar function $\phi$ stands for the gravitational potential, while $\chi(c)$ and $f(c)$ are the chemotactic sensitivity and the per-capita oxygen consumption rate that may depend on $c$. $\kappa\in\mathbb{R}$ is a parameter such that if $\kappa\neq 0$, the fluid motion is governed by the  Navier-Stokes equation, while $\kappa=0$, the equation for $u$ is simplified to be the Stokes equation. We complement system \eqref{chemo1} with the following Neumann and no-slip boundary conditions
\beq\label{boundcon}\frac{\partial n}{\partial\nu}=\frac{\partial c}{\partial \nu}=0\quad\text{and}\quad u=0\quad\text{for} \;\; x\in {\partial\Omega},\; t>0\eeq
with $\nu$ being the unit outward normal to $\partial\Omega$, and the initial conditions
\beq\label{ini} n(x,0)=n_0(x),\quad c(x,0)=c_0(x),\quad u(x,0)=u_0(x).\eeq

The chemotaxis-fluid system has recently been extensively studied by many authors. Here, we only mention some contributions in the literature that are mostly related to our initial boundary value problem \eqref{chemo1}--\eqref{ini} in a bounded domain $\Omega\subset\mathbb{R}^N$. We refer the reader to \cite{LL2011,DLM2010,CKL2013,CKLCPDE2013,TZJMAA} for the studies on the Cauchy problem in the whole space $\mathbb{R}^N$, and to \cite{FLM2010,LL2011,TW2013} for the case that the linear diffusion term $\Delta n$ is replaced by a nonlinear porous medium type one $\Delta n^m$. When the domain $\Omega$ is bounded and regular, Lorz \cite{Lorz2010} proved the existence of a local weak solution by Schauder's fixed point theory when $N=3$. If the domain $\Omega$ is further assumed to be {\em convex}, then using a key observation made in \cite{TWDCDS2012} and some delicate entropy-energy estimates, Winkler \cite{WinklerCPDE} established the existence of a unique global classical solution with large initial data for $\kappa\in\mathbb{R}$ when $N=2$, and the existence of a global weak solution for with $\kappa=0$ when $N=3$. Later, in \cite{WinkARMA}, the same author further proved that the global classical solution obtained in \cite{WinklerCPDE} in 2D will converge to a constant state $(n_\infty, 0,0)$ as times goes to infinity.
The main purpose of this note is to extend the results by Winkler from the convex domain to the general smooth domain.

Throughout this paper, we denote by $L^q(\Omega)$, $W^{k,q}(\Omega)$, $1\leq q\leq\infty$, $k\in\mathbb{N}$  the usual Lebesgue and Sobolev spaces respectively, and as usual, $H^k(\Omega)=W^{k,2}(\Omega)$.  $\|\cdot\|_{B}$ denotes the norm in the space $B$; we also use the abbreviation $\|\cdot\|:=\|\cdot\|_{L^2(\Omega)}$.

We introduce the following assumptions on the initial data as in \cite{WinklerCPDE}:
\beq
\begin{cases}\label{ini1}
n_0\in C^0(\overline{\Omega}), \quad \ n_0>0\ \text{in}\ \overline{\Omega},\\
c_0\in W^{1,q}(\Omega),\ \  \text{for some }\ q>N,\ \ c_0>0\ \text{in}\ \overline{\Omega},\\
u_0\in D(A^{\alpha}),\quad\, \text{for some}\ \ \alpha\in(\frac{N}{4},1),
\end{cases}
\eeq
where $A$ denotes the realization of the Stokes operator in the solenoidal subspace $L^2_{\sigma}(\Omega)$ that is the closure of $\{u\in (C_0^{\infty}(\Omega))^N; \ \nabla \cdot u=0\}$ in $(L^2(\Omega))^N$ (see e.g., \cite{G}). As for the parameter functions, we suppose that
\beq
\begin{cases}\label{ass1}
\chi\in C^2[0,\infty), \quad \chi>0 \text{  in  }[0,\infty),\\
f\in C^2[0,\infty), \quad f(0)=0,\ f>0 \text{  in  }(0,\infty),\\
\phi\in C^2(\overline{\Omega}),
\end{cases}
\eeq
and
\beq
\left(\frac{f}{\chi}\right)'>0,\quad \left(\frac{f}{\chi}\right)''\leq 0,\quad \left(\chi\cdot f\right)'\geq 0,\quad  \text{on  }[0,\infty).\label{ass4}
\eeq

Now we state our main result of the paper:
\begin{theorem}\label{MT}
Let $\Omega\in \mathbb{R}^N$ $(N\in\{2,3\})$ be a bounded regular domain. Assume that the assumptions \eqref{ini1}--\eqref{ass4} are satisfied.
\begin{itemize}
\item[(i)] If $N=2$ and $\kappa\in \mathbb{R}$, then problem \eqref{chemo1}--\eqref{ini} admits a unique global classical solution $(n,c,u,p)$ (up to addition of a constant to the pressure $p$) such that for any $T>0$,
\beq
\begin{cases}
n\in C([0,T]; L^2(\Omega))\cap L^{\infty}(0,T;C^0(\overline{\Omega}))\cap C^{2,1}(\overline{\Omega}\times(0,T)),\\
c\in C([0,T]; L^2(\Omega))\cap L^{\infty}(0,T;W^{1,q}(\Omega))\cap C^{2,1}(\overline{\Omega}\times(0,T)),\\
u\in C([0,T];L^2(\Omega))\cap L^{\infty}(0,T;D(A^{\alpha}))\cap C^{2,1}(\overline{\Omega}\times(0,T)),\\
p\in L^1(0,T;H^1(\Omega)).
\end{cases}\nonumber
\eeq
Moreover, the global classical solution satisfies
\beq
\lim_{t\to +\infty} \|n(\cdot, t)-n_\infty\|_{L^\infty(\Omega)}+\|c(\cdot, t)\|_{L^\infty(\Omega)}+\|u(\cdot, t)\|_{L^\infty}=0,\label{conver}
\eeq
where $n_\infty=\frac{1}{|\Omega|}\int_\Omega n_0 dx$.
\item[(ii)] If $N=3$ and $\kappa=0$, then there exists at least one global weak solution of \eqref{chemo1}--\eqref{ini} in the sense of \cite[Definition 5.1]{WinklerCPDE}.
%\item If $N=3$ and $\kappa\neq0,$ then there exists at least one global weak solution of \eqref{chemo1}-\eqref{ini} in the sense of Definition \ref{def1}.
\end{itemize}
\end{theorem}

Before giving the detailed proof, we stress some new features of the present paper. Our initial boundary problem \eqref{chemo1} is proposed on any regular bounded domain $\Omega\subset \mathbb{R}^N$ $(N\in\{2,3\})$, while in \cite{WinklerCPDE} convexity of the domain is essentially used in order to prove the global existence results. The main difficulty comes from an integration term on the boundary that takes the following form:
\beq
\int_{\partial\Omega}\frac{\chi(c)}{f(c)}\frac{\partial}{\partial\nu}|\nabla c|^2 dS.\nonumber
\eeq
If $\Omega$ is convex, then the above integrand turns out to be non-positive (cf. \cite{PGG1998}) and as a consequence it can be simply neglected like in \cite{WinklerCPDE}. However, if we consider the problem in a general bounded domain, this integrand fails to have a definite sign and has to be estimated in a suitable way. To overcome this difficulty, we make use of the elementary lemma due to Mizoguchi \& Souplet \cite{MizoSoup} together with the trace theorem to control the above boundary integration. Then by delicate estimates, we are able to derive a new type of entropy-energy estimate, whose right-hand side terms turns out to be easier to handle (cf. \eqref{e1} below). This estimate plays a key role in obtaining global existence results for problem \eqref{chemo1}--\eqref{ini} in a similar manner as in \cite{WinklerCPDE}. Moreover, based on it we can further prove that when $N=2$, the global classical solution will converge to a constant steady state $(n_\infty, 0,0)$ as times goes to infinity. Our work removes the restriction on the convexity of $\Omega$ and thus improves the corresponding results by Winkler \cite{WinklerCPDE, WinkARMA}.

The rest of this paper is organized as follows. In Section 2, we handle the above mentioned boundary integration term with the aid of a gradient estimate result on the boundary from Mizoguchi \& Souplet \cite{MizoSoup}. In Section 3, we derive the key entropy-energy estimate by using some delicate calculations and then sketch the proof for Theorem \ref{MT}.

\section{Estimate of Boundary Integration}
To begin with, we first state two basic properties of the solutions to problem \eqref{chemo1}--\eqref{ini} that have already been obtained in \cite{WinklerCPDE,WinkARMA,LL2011}.
\begin{lemma} \label{lla}
We have
\beq
\int_{\Omega}n(x,t)dx=\int_{\Omega}n_0dx\quad\text{for all    }t>0
\eeq
and
\beq
t\mapsto\|c(\cdot,t)\|_{L^{\infty}(\Omega)} \quad\text{is non-increasing.}
\eeq
In particular,
\beq
\|c(\cdot, t)\|_{L^{\infty}(\Omega)}\leq\|c_0\|_{L^{\infty}(\Omega)}\quad\text{for all    }t>0.
\eeq
\end{lemma}

%\begin{lemma} Suppose that $h\in C^2(\mathbb{R})$. Then for all $\varphi\in C^2(\overline{\Omega})$ fulfilling $\frac{\partial\varphi}{\partial\nu}=0$ on $\partial\Omega$, we have
%\beq\begin{split} \int_{\Omega}h'(\varphi)|\nabla\varphi|^2\Delta\varphi=&-\frac{2}{3}\int_{\Omega}h(\varphi)|\Delta\varphi|^2
%+\frac{2}{3}\int_{\Omega}h(\varphi)|D^2\varphi|^2-\frac{1}{3}\int_{\Omega}h''(\varphi)|\Delta\varphi|^4\\
%&-\frac{1}{3}\int_{\partial\Omega}h(\varphi)\frac{\partial}{\partial\nu}|\Delta\varphi|^2\end{split}\eeq where $D^2\varphi$ denotes the Hessian of $\varphi.$
%\end{lemma}

\begin{lemma} The solution satisfies the identity
\beq
\begin{split} \label{iden}
\frac{d}{dt}&\left\{\int_{\Omega}n\log n dx+\frac{1}{2}\int_{\Omega}|\nabla\psi(c)|^2dx\right\}+\int_{\Omega}\frac{|\nabla n|^2}{n}dx+\int_{\Omega}g(c)|D^2\rho(c)|^2dx\\
&=-\frac{1}{2}\int_{\Omega}\frac{g'(c)}{g^2(c)}|\nabla c|^2(u\cdot \nabla c)dx+\int_{\Omega}\frac{1}{g(c)}\Delta c(u\cdot\nabla c)dx\\
&+\int_{\Omega}F(n)\left(\frac{f(c)g'(c)}{2g^2(c)}-\frac{f'(c)}{g(c)}\right)|\nabla c|^2dx\\
&+\frac{1}{2}\int_{\Omega}\frac{g''(c)}{g^2(c)}|\nabla c|^4dx+\frac{1}{2}\int_{\partial\Omega}\frac{1}{g(c)}\frac{\partial}{\partial\nu}|\nabla c|^2dx
 \end{split}
 \eeq
 for all $t\in(0, T_{max})$, where $D^2\rho$ denotes the Hessian of $\rho$ and we have set
 \beq\label{gpr}
 g(s):=\frac{f(s)}{\chi(s)},\ \ \ \psi(s):=\int_1^s\frac{d\sigma}{\sqrt{g(\sigma)}}\ \ \ \ \text{and}\ \ \ \rho(s):=\int_1^s\frac{d\sigma}{g(\sigma)} \ \ \ \text{for  }s>0,
 \eeq
 and $F\in C^2([0,+\infty))$ is a nonnegative function satisfying $0\leq F'(s)\leq 1$ for all $s\geq 0$.
\end{lemma}

Therefore, if the domain $\Omega$ is convex and $\frac{\partial{c}}{\partial \nu}=0$ on $\partial \Omega$, then $\frac{\partial}{\partial \nu}|\nabla c|^2\leq 0$ (cf. \cite{PGG1998}). This implies that the boundary integration term on the right-hand side of \eqref{iden} is non-positive and hence can be simply neglected as in \cite{WinklerCPDE}. However, for general bounded domains, it fails to have a definite sign and thus needs to be estimated in a suitable way.

For this purpose, we introduce a lemma by Mizoguchi \& Souplet \cite[Lemma 4.2]{MizoSoup}, which enables us to deal with the boundary integration in \eqref{iden}.
\begin{lemma} \label{MS} For the bounded domain $\Omega$ and $w\in C^2(\overline{\Omega})$ satisfying $\frac{\partial w}{\partial \nu}=0$ on $\partial\Omega$. We have
\beq \frac{\partial|\nabla w|^2}{\partial\nu}\leq2\kappa|\nabla w|^2\quad\text{on   }\partial\Omega,\eeq
where $\kappa=\kappa(\Omega)>0$ is an upper bound for the curvatures of $\partial\Omega.$
\end{lemma}

Then we can prove the following boundary estimate:

\begin{lemma} \label{BDES}
Suppose that the assumptions of Theorem \ref{MT} hold, and let $g,\psi$ and $\rho$ be defined as in \eqref{gpr}. Then for any $\epsilon\in (0,1)$, the following estimate holds
\beq
\begin{split}
&\frac{1}{2}\left|\int_{\partial\Omega}\frac{1}{g(c)}\frac{\partial}{\partial\nu}|\nabla c|^2dS\right| \\
&\ \ \leq \e\int_{\Omega}g(c)|\Delta\rho(c)|^2dx+\e\int_{\Omega}\frac{|g'(c)|^2|\nabla c|^4}{g(c)^3}dx+C_{\e}\|\psi(c)\|^2,
\end{split}\label{bde1}
\eeq
where $C_\e$ is a constant that may depend on $\Omega$, $\kappa$ and $\e$, but not on $c$.
\end{lemma}
\begin{proof}
Thanks to Lemma \ref{MS}, we can control the boundary integration term as follows
\beq
 \frac{1}{2}\left|\int_{\partial\Omega}\frac{1}{g(c)}\frac{\partial}{\partial\nu}|\nabla c|^2dS \right|
 \leq
 \kappa(\Omega)\int_{\partial\Omega}\frac{|\nabla c|^2}{g(c)}dS
 =\kappa(\Omega)\int_{\partial\Omega}|\nabla \psi(c)|^2dS.
 \label{bde2}
 \eeq
On the other hand, by the trace theorem \cite[Theorem I.9.4]{LM}, it holds
\beq \int_{\partial\Omega}|\nabla \psi(c)|^2 dS\leq C\|\psi(c)\|_{H^{\frac{3+s}{2}}(\Omega)}^2,
\quad \forall\, s\in (0, \frac12).\label{bde3}
\eeq
where $C>0$  depends only on $\Omega$. Then we infer from the interpolation inequality (see e.g., \cite[Remark I.9.6]{LM}) that
\beq
\begin{split}
\|\psi(c)\|_{H^{\frac{3+s}{2}}(\Omega)}^2&\ \ \leq C\|\psi(c)\|_{H^2(\Omega)}^\frac{3+s}{2}\|\psi(c)\|^\frac{1-s}{2}\\
&\ \ \leq C\|\Delta\psi(c)\|^\frac{3+s}{2}\|\psi(c)\|^\frac{1-s}{2}+C\|\psi(c)\|^2
\end{split}
\label{inter}
\eeq
By direct calculations, we have
\beq \label{id2}
\Delta\psi(c)=\frac{\Delta c}{\sqrt{g(c)}}-\frac{1}{2}\frac{g'(c)|\nabla c|^2}{(g(c))^{\frac{3}{2}}}
\eeq
and
\beq
\Delta\rho(c)=\frac{\Delta c}{g(c)}-\frac{g'(c)|\nabla c|^2}{g(c)^2},\non
\eeq
which yield that
\beq \sqrt{g(c)}\Delta\rho(c)=\frac{\Delta c}{\sqrt{g(c)}}-\frac{g'(c)|\nabla c|^2}{(g(c))^{\frac{3}{2}}}=\Delta\psi(c)-\frac{1}{2}\frac{g'(c)|\nabla c|^2}{(g(c))^{\frac{3}{2}}}.\eeq
Then it follows that
\beq
\label{boun1}
\|\Delta\psi(c)\|^2\leq2\int_{\Omega}g(c)|\Delta\rho(c)|^2dx+\frac{1}{2}\int_{\Omega}\frac{|g'(c)|^2|\nabla c|^4}{g(c)^3}dx.
\eeq
Using the estimates \eqref{bde2}--\eqref{inter}, \eqref{boun1} and Young's inequality, we can conclude \eqref{bde1}. The proof is complete.
\end{proof}

\section{Proof of Theorem \ref{MT}}

Lemma \ref{BDES} enables us to prove the following entropy-energy type estimate for problem \eqref{chemo1}--\eqref{ini}:

\begin{lemma} Suppose that the assumptions of Theorem \ref{MT} hold, and let $g,\psi$ and $\rho$ be defined as in \eqref{gpr}. Then the smooth solution to problem \eqref{chemo1}--\eqref{ini} satisfies
\beq
\begin{split}
& \frac{d}{dt}\left\{\int_{\Omega}n\log n dx+\frac{1}{2}\int_{\Omega}|\nabla\psi(c)|^2dx \right\}
+\int_{\Omega}\frac{|\nabla n|^2}{n}dx +\frac{1}{2}\int_{\Omega}g(c)|D^2\rho(c)|^2dx \\
&\ \ \leq C\|\nabla u\|^2+C\|\psi(c)\|^2\quad \text{for all}\ t\in(0,T_{max}),
\end{split}
\label{e1}
\eeq
where  $C>0$ is independent of $t$, $T_{max}$.
\end{lemma}
\begin{proof}
We proceed to estimate the right-hand side of \eqref{iden} term by term. First, applying \cite[Lemma 3.3]{WinklerCPDE}, we obtain that
\beq
\int_{\Omega}\frac{g'(c)}{g(c)^3}|\nabla c|^4dx \leq(2+\sqrt{N})^2\int_{\Omega}\frac{g(c)}{g'(c)}|D^2\rho(c)|^2dx.
\label{boun2}
\eeq
Since $g'>0$ on $[0,+\infty)$ and $0\leq c\leq K:=\|c_0\|_{L^{\infty}}$ (see Lemma \ref{lla}), there exist constants
$$
C_1:=\inf_{s\in(0,K)}g'(s)>0,\quad \text{and} \quad C_2=\sup_{s\in(0,K)}g'(s)>0
$$
such that $C_1\leq g'(c)\leq C_2$ in $\Omega\times(0,T_{max})$. Hence, it follows from \eqref{boun2} that
\beq
\label{boun3}
\int_{\Omega}\frac{|g'(c)|^2}{g(c)^3}|\nabla c|^4dx\leq\frac{C_2}{C_1}(2+\sqrt{N})^2\int_{\Omega}g(c)|D^2\rho(c)|^2dx.
 \eeq

Then for the first two terms on the right-hand side of \eqref{iden}, using \eqref{id2}, after integration by parts, and using the boundary conditions \eqref{boundcon}, we deduce from \eqref{boun3} that
\beq
\begin{split}
&-\frac{1}{2}\int_{\Omega}\frac{g'(c)}{g(c)^2}|\nabla c|^2(u\cdot \nabla c)dx +\int_{\Omega}\frac{1}{g(c)}\Delta c(u\cdot\nabla c)dx \\
&\ \ =\int_\Omega \frac{1}{\sqrt{g(c)}}\Delta\psi(c)(u\cdot\nabla c) dx\\
&\ \ =\int_{\Omega}\Delta\psi(c)(u\cdot \nabla\psi(c))dx\\
&\ \ =-\int_{\Omega}(\nabla\psi(c)\otimes\nabla\psi(c)): \nabla u dx \\
&\ \ \leq \frac{C_1}{4C_2(2+\sqrt{N})^2}\int_{\Omega}\frac{g'(c)}{g(c)^3}|\nabla c|^4dx +\frac{C_2(2+\sqrt{N})^2}{C_1}\int_{\Omega}\frac{g(c)}{g'(c)}|\nabla u|^2dx\\
&\ \ \leq \frac14\int_{\Omega}g(c)|D^2\rho(c)|^2dx+\frac{C_2(2+\sqrt{N})^2}{C_1}\int_{\Omega}\frac{g(c)}{g'(c)}|\nabla u|^2dx,
\end{split}\label{bde2a}
\eeq
where we have also used the incompressibility of the fluid as well as Young's inequality. Here, the notion $A : B$ denotes $\mathrm{Tr}(AB) = A_{ij}B_{ji}$ for two $N\times N$ matrices $A, B$.

As in \cite{WinklerCPDE}, it follows from the assumption \eqref{ass4} that $g''\leq 0$ in $[0, +\infty)$ and $\frac{fg'}{2g^2}-\frac{f'}{g}=-\frac{(\chi f)'}{2f}\leq 0$ on $[0, +\infty)$. Then the third term on the right-hand side of \eqref{iden} is indeed non-positive and can simply be neglected.

Next, we infer from Lemma \ref{BDES} that the last boundary integration term on the right-hand side of \eqref{iden} can be estimated by using \eqref{bde1}. In view of the pointwise inequality $|\Delta z|^2\leq N|D^2 z|^2$ for $z\in C^2(\overline{\Omega})$, the first term on the right-hand side of \eqref{bde1} can be estimated as follows:
\beq
\int_{\Omega}g(c)|\Delta\rho(c)|^2dx\leq  N \int_{\Omega}g(c)|D^2\rho(c)|^2dx.\label{boun1b}
\eeq
As a consequence, it follows from \eqref{bde1}, \eqref{boun3} and  \eqref{boun1b}  that
\beq
\begin{split}
&\frac{1}{2}\left|\int_{\partial\Omega}\frac{1}{g(c)}\frac{\partial}{\partial\nu}|\nabla c|^2dS\right| \\
&\ \ \leq \e\left(N+\frac{C_2}{C_1}(2+\sqrt{N})^2\right)\int_{\Omega}g(c)|D^2\rho(c)|^2dx+C_{\e}\|\psi(c)\|^2.
\end{split}\label{bde1a}
\eeq
Taking $\e>0$ sufficiently small such that $\e\left(N+\frac{C_2}{C_1}(2+\sqrt{N})^2\right)\leq \frac{1}{4}$ in \eqref{bde1a}, we can deduce from \eqref{iden}, \eqref{bde2a} and \eqref{bde1a} that the inequality \eqref{e1} holds. The proof is complete.
\end{proof}

Besides, we can derive the following estimates on the velocity field $u$:

\begin{lemma} Under the assumptions of Theorem \ref{MT}, we have the following estimate
 \beq
 \label{bound4}
 \|u\|^2+\int_0^T\|\nabla u\|^2dt\leq C\left(\int_0^T\int_{\Omega}\frac{|\nabla n|^2}{n}dx+1\right)^{\frac{N}{6}},\;\;\text{for any}\ 0<T<T_{max},
 \eeq
 where the constant $C$ depends on $T$ and $\int_\Omega n_0 dx$.
\end{lemma}
\begin{proof}
Multiplying the third equation in \eqref{chemo1} by $u$, integrating over $\Omega$, using the H\"older inequality and the Sobolev embedding theorem, we obtain that
\beq
\begin{split}
\frac{1}{2}\frac{d}{dt}\int_{\Omega}|u|^2dx +\int_{\Omega}|\nabla u|^2dx
& =\int_{\Omega} n\nabla\phi\cdot u dx \\
& \leq \|u\|_{L^6}\|\nabla \phi\|_{L^\infty}\|n\|_{L^\frac65}\\
& \leq \frac{1}{2}\|\nabla u\|^2+C\|\nabla\phi\|_{L^{\infty}}^2\|n\|_{L^{\frac{6}{5}}}^2.\label{aaa}
\end{split}
\eeq
On the other hand, we infer from \cite[Lemma 4.1]{WinklerCPDE} that
\beq
\int_0^T\|n\|_{L^{\frac{6}{5}}}^2dt\leq C(T, \int_\Omega n_0 dx)\left(\int_0^T\int_{\Omega}\frac{|\nabla n|^2}{n}dxdt+1\right)^{\frac{N}{6}}.\label{aaaa}
\eeq
Then integrating \eqref{aaa} with respect to time, we conclude from \eqref{aaaa} that the assertion \eqref{bound4} follows.
\end{proof}

As a consequence of the above lemmas, we can obtain the following \emph{a priori} estimates on solutions to problem \eqref{chemo1}--\eqref{ini}:

\begin{lemma} \label{lemkey} Under the same assumptions of Theorem \ref{MT},  for any $0<T<T_{max}$, we have the following estimates
\beq\label{bound5a}
\int_{\Omega}n\log n dx+\int_{\Omega}|\nabla\psi(c)|^2dx+\int_{\Omega}|u|^2dx\leq C,
\eeq
\beq \int_0^T\int_{\Omega}\frac{|\nabla n|^2}{n}dxdt+\int_0^T\int_{\Omega}g(c)|D^2\rho(c)|^2dxdt+\int_0^T\int_{\Omega}|\nabla u|^2dxdt\leq C,\label{bound5b}
\eeq
and
\beq
\label{bound6}
\int_0^T\int_{\Omega}|\nabla c|^4 dxdt \leq C.
\eeq
\end{lemma}
\begin{proof} The estimates \eqref{bound5a}--\eqref{bound5b} easily follow from \eqref{e1} and \eqref{bound4} together with an application of Young's inequality. Then the estimate \eqref{bound6} can be derived in the same way as in \cite[Corollary 4.4, Corollary 5.3]{WinklerCPDE}.
\end{proof}

After the previous preparations, we can proceed to prove our main result.

\textbf{Proof of Theorem \ref{MT}}. The proof for the existence of global solutions follows a similar argument like in \cite{WinklerCPDE}. First, since the proof in \cite{WinklerCPDE} for the local existence and uniqueness of classical solutions to problem \eqref{chemo1}--\eqref{ini} does not rely on the fact that $\Omega$ is convex, then for the current case with general domain $\Omega$, we can still apply the same fixed point argument to prove the local wellposeness result.  Then by the \emph{a priori} estimates in Lemma \ref{lemkey}, we are able to prove the existence of global classical solutions for $N=2$ and global weak solutions for $N=3$ in the same way as in \cite[Section 4, Section 5]{WinklerCPDE}. Therefore, the details are omitted here.

Concerning the asymptotic behavior for $N=2$, we notice that in our new entropy-energy inequality \eqref{e1}, the right-hand side terms now become $C\|\nabla u\|^2+C\|\psi(c)\|^2$ instead of $C\int_\Omega |u|^4 dx$ as in \cite[(3.11)]{WinklerCPDE} (see also \cite[(2.15)]{WinkARMA}). In \cite{WinkARMA} this (worse) term $\int_\Omega |u|^4 dx$ is estimated by using the Sobolev embedding theorem and Poincar\'e's inequality such that
$$\|u\|_{L^4(\Omega)}^4\leq C\|\nabla u\|^2\|u\|^2,$$
which is a higher-order nonlinearity than $\|\nabla u\|^2$. Since the current right-hand side of \eqref{e1} is simpler and the second term $\|\psi(c)\|^2$ is uniformly bounded in time due to the definition of $\psi$ (cf. \eqref{gpr}) and Lemma \ref{lla}, one can check that all the arguments in \cite{WinkARMA} can pass through. As a consequence, we can conclude the convergence result \eqref{conver}.

Hence, Theorem \ref{MT} is proved.

%We remark that since our entropy-energy estimate \eqref{e1} improves that obtained in \cite{WinklerCPDE}, the argument in 2D can be simplified to some extent.

\bigskip

\noindent \textbf{Acknowledgments:}
J. Jiang is partially supported by National Natural Science Foundation of China (NNSFC) under the grants No. 11201468 and No. 11275259, H. Wu is partially supported by NNSFC under the grant No. 11371098 and Zhuo Xue program in Fudan University, and S. Zheng is partially supported by NNSFC under the grant No. 11131005.


\begin{thebibliography}{99}
 \addtolength{\itemsep}{-1ex}
%\bibitem{Amann2} H. Amann, Dynamic theory of quasilinear parabolic equations II: Reaction-diffusion systems, Differential and Integral Equations, \textbf{3}(1990), 13--75.
%
%\bibitem{Amann3} H. Amann, Dynamic theory of quasilinear parabolic systems III: Global existence, Math. Z., \textbf{202}(1989), 219--250.

\bibitem{CKL2013} M. Chae, K. Kang and J. Lee, Existence of smooth solutions to coupled chemotaxis fluid equations, Discrete Contin. Dyn. Syst., \textbf{33}(2013), 2271--2297.

\bibitem{CKLCPDE2013} M. Chae, K. Kang and J. Lee, Global existence and temporal decay in Keller-Segel models coupled to fluid equations, Commun. Partial Diff. Equ., \textbf{39}(2014), 1205--1235.

\bibitem{PGG1998} R. Dal Passo, H. Garcke and G. Gr\"{u}n, On a fourth-order degenerate parabolic equation: global entropy estimates, existence, and qualitative behavior of solutions, SIAM J. Math. Anal., \textbf{29}(1998), 321--342.

%\bibitem{FZ2013} J.S. Fan and K. Zhao, Global dynamics of a coupled chemotaxis-fluid model on bounded domains, J. Math. Fluid Mech., DOI: 10.1007/s00021-013-0162-1, 2013.

\bibitem{FLM2010} M. Di Francesco, A. Lorz and P.A. Markowich, Chemotaxis-fluid coupled model for swimming bacteria with nonlinear diffusion: global existence and asymptotic behavior, Discrete Contin. Dyn. Syst., \textbf{28}(2010), 1437--1453.

\bibitem{DLM2010} R. Duan, A. Lorz and P.A. Markowich, Global solutions to the coupled chemotaxis-fluid equations, Commun. Partial Diff. Equ., \textbf{35}(2010), 1635--1673.

\bibitem{G} Y. Giga, Domains of fractional powers of the Stokes operator in $L_r$ spaces, Arch. Rational Mech. Anal., \textbf{89}(1985), 251--265.
%
%\bibitem{GM} Y. Giga and T. Miyakawa, Solutions in $L_r$ of the Navier-Stokes initial value problem, Arch. Ration. Mech. Anal., \text{85}(1985), 267--281.

\bibitem{LL2011} J.-G. Liu and A. Lorz, A coupled chemotaxis-fluid model: global exsitence, Ann. Inst. H. Poincar\'{e} Anal. Non Li\'{e}naire, \textbf{28}(2011), 643--652.


\bibitem{LM} J.-L. Lions and E. Magenes, Non-homogeneous Boundary Value Problems and Applications, Vol. I. Springer-Verlag, New York, 1972.

\bibitem{Lorz2010} A. Lorz, Coupled Chemotaxis fluid model, Math. Mod. Meth. Appl. Sci., \textbf{20}(2010), 987--1004.

%\bibitem{Miya} T. Miyakawa, On the initial value problem for the Navier-Stokes
%equations in $L^p$ spaces, Hiroshima Math. J., \textbf{11}(1981), 9--20.

\bibitem{MizoSoup} N. Mizoguchi and Ph. Souplet, Nondegeneracy of blow-up points for the parabolic Keller-Segel system, Ann. I. H. Poincar\'{e}-AN, 2014, in press.

\bibitem{TZJMAA}Z. Tan and X. Zhang, Decay estimates of the coupled chemotaxis-fluid equations in $\mathbb{R}^3$, J. Math. Anal. Appl., \textbf{410}(2014), 27--38.

\bibitem{TWDCDS2012} Y.-S. Tao and M. Winkler, Global existence and boundedness in a Keller-Segel-Stokes model with arbitrary porous medium diffusion, Discrete Contin. Dyn. Syst., \textbf{32}(2012), 1901--1914.

\bibitem{TW2013} Y.-S. Tao and M. Winkler, Locally bounded global solutions in a three-dimensional chemotaxis-Stokes system with nonlinear diffusion, Ann. I. H. Poincar\'{e}-AN, \textbf{30}(2013), 157--178.



\bibitem{PNAS2005} I. Tuval, L. Cisneros, C. Dombrowski, C.W. Wolgemuth, J.O. Kessler and R.E. Goldstein, Bacterial swimming and oxygen tansport near contact lines, Proc. Natl. Acad. Sci. USA, \textbf{102}(2005), 2277--2282.

\bibitem{WinklerCPDE} M. Winkler, Global large data solutions in a chemotaxis-Navier-Stokes system modeling cellular swimming in fluid drops, Commun. Partial Diff. Equ., \textbf{37}(2012), 319--351.

\bibitem{WinkARMA} M. Winkler, Stabilization in a two-dimensional chemotaxis-Navier-Stokes system, Arch. Rational Mech. Anal., \textbf{211}(2014), 455--487.











\end{thebibliography}
\end{document}